\documentclass[11pt]{article}
\usepackage{amsmath,mathrsfs}
\usepackage{amssymb}
\usepackage{color}
\usepackage{amsthm}
\usepackage{indentfirst}
\allowdisplaybreaks

\newtheorem{remark}{Remark}
\newtheorem{theorem}{Theorem}
\newtheorem*{theorem*}{Theorem}
\newtheorem{corollary}{Corollary}
\newtheorem{lemma}{Lemma}
\newtheorem{definition}{Definition}

\newtheorem*{proposition*}{Proposition}

\newcommand{\be}{\begin{eqnarray*}}
\newcommand{\ee}{\end{eqnarray*}}
\newcommand{\ba}{\begin{align*}}
\newcommand{\bpm}{\begin{pmatrix}}
\newcommand{\epm}{\end{pmatrix}}

\begin{document}
\title{Power-bounded quaternionic operators}


\author{Qinghai Huo\thanks{E-mail:  {\tt hqh86@mail.ustc.edu.cn}. } \ and Zhenghua Xu\thanks{E-mail: {\tt zhxu@hfut.edu.cn}.
 This work  is partially  supported by the  National Natural Science Foundation of China (No.  12301097) and the  Anhui Provincial Natural Science Foundation (No. 2308085MA04).}   \\  \emph{\small School of Mathematics, Hefei University of Technology,}\\ \emph{\small Hefei, 230601, P.R. China} \\~\\}
\date{}

\maketitle

\begin{abstract}
 Recently, the conception  of slice regular functions  was allowed to  introduce   a new quaternionic functional calculus, among which
the theory of semigroups of linear operators  was  developed  into the quaternionic setting, even in a more general   case of real alternative $*$-algebras. In this paper,   we  initiate  to study the discrete case  and introduce the notion of  power-bounded quaternionic operators.  In particular, by the spherical Yosida approximation,  we establish a discrete Hille-Yosida-Phillips theorem  to give an equivalent characterization of quaternionic linear operators  being power-bounded. A sufficient condition of the power-boundedness for quaternionic linear operators is also given. In addition, a  non-commutative version of  the Katznelson-Tzafriri theorem (J. Funct. Anal. 68: 313-328, 1986) for  power-bounded quaternionic operators  is formulated  in terms of the $S$-spectrum.
\end{abstract}
{\bf Keywords:}\quad $S$-spectrum; $S$-resolvent operator; functions of a hypercomplex variable; power-bounded quaternionic operators \\
{\bf MSC (2020):}\quad Primary: 47A10;  Secondary: 30G35, 47A25, 47A60
\section{Introduction}
Given a complex Banach space $X$,  denote by $\mathcal{B}(X)$  the space of bounded linear operators on $X$.  A (complex) linear operators semigroup on $X$  is a mapping $T :[0, \infty) \rightarrow  \mathcal{B}(X)$ such that $T(0)$ is the identity operator and
\begin{equation}\label{semigroup-property}
T(t+s) = T(t)T(s), \quad \forall\ t,s >0.\end{equation}
The  semigroup $T(t)$ is called strongly continuous if
$$\lim_{t \rightarrow 0^{+}}  T(t)x=x, \quad \forall  \ x\in X.$$

For more  basic theories and applications of operator semigroups, we refer to monographs e.g. \cite{Davies,Engel-Nagel,Pazy}. The  following  generation theorem is a fundamental result in the theory of one-parameter semigroups of linear operators  (cf. \cite[p. 19]{Pazy}).

 \begin{theorem}[Hille-Yosida-Phillips]\label{Hille-Yosida-Phillips}
  A  complex  linear operator $A$ is the infinitesimal  generator of a strongly continuous semigroup $T(t)$, satisfying $\| T(t)\|\leq C$ for all $t> 0$,  if and only if

(i)  $A$ is closed  and densely defined,

(ii) the  resolvent set  $\varrho(A)$ of $A$ contains $\mathbb{R}^{+}=(0,+\infty)$ and
$$\|R_{\lambda}(A)^{n}\|\leq \frac{C}{\lambda^{n}}, \quad \lambda\in \mathbb{R}^{+}, n\in \mathbb{N},$$
where $C$ is constant,  $R_{\lambda}(A)=(\lambda \mathcal{I}-A)^{-1}$ is the  resolvent of $A$, and $\mathbb{N}$ denotes the set of all positive integers.
 \end{theorem}

By introducing the so-called  \textit{$S$-resolvent operator} and \textit{$S$-spectrum} instead of classical resolvent operator and spectrum respectively, the generation theorem above relating generators and semigroups was recently  generalized  into the quaternionic case \cite{Colombo-Sabadini-11} by Colombo and Sabadini and into the general case of real alternative $*$-algebras \cite{Ghiloni} by Ghiloni and Recupero.

Due to the  non-commutation of quaternions, there exists naturally left spectrum and right spectrum  for quaternionic  right linear operators on a two-sided  quaternionic Banach space. However, this strategy is not going very well, especially, when considering  quaternionic  analogues of holomorphic functional calculus.  The discovery of $S$-spectrum and $S$-resolvent operator   in 2006 is a breakthrough, which benefits from the development of the theory of slice regular functions \cite{Gentili-Struppa-07},   and  shed light on possible formulations  of quaternionic hyper-holomorphic functional calculus. Nowadays, the notions of $S$-spectrum and $S$-resolvent operator are proved to be powerful and got fruitful  investigations; see e.g. \cite{Alpay-Colombo-Sabadini,Colombo-Gantner-19,Colombo-Gantner-Kimsey-18,Colombo-Gantner-Kimsey-Sabadini-22,Colombo-Sabadini-Struppa-11,Ghiloni-Moretti-Perotti-13,
Wang-Qin-Agarwal} and references therein.

For $T\in \mathcal{B}(X)$, the family $\{T^{n}\}_{n\in \mathbb{N}}$ can be viewed as  a discrete semigroup satisfying the condition $T^{m+n}=T^{m}T^{n}$ for all $m,n\in \mathbb{N}$, which is the discrete version of semigroup property in (\ref{semigroup-property}).  In this paper, we consider  the discrete version  of one-parameter semigroups of quaternionic  linear operators  by introducing the notion of \textit{power-bounded quaternionic operators}.

The remaining part of the paper is organized as follows. The next section is devoted to some preliminary notions on right quaternionic linear operators and  quaternionic $S$-functional calculus. In Section 3, we introduce the concept of power-bounded quaternionic operators and discuss  their properties. In particular, we give a discrete analog of Theorem \ref{Hille-Yosida-Phillips} for power-bounded quaternionic operators, which generalizes a result of  Nevanlinna into the quaternionic setting (Theorem \ref{Nevanlinna-H}) by the notion of the \textit{spherical Yosida approximation}.  In addition, a sufficient condition of   the power-boundedness for  quaternionic operators is given in  Theorem \ref{sufficient-condition}. The main tool of our results is proved to be relying  on   the $S$-spectrum and $S$-resolvent operator. Finally, in Section 4,   the $S$-spectrum  can be used   to build a quaternionic analogue of the  celebrated Katznelson-Tzafriri theorem for power-bounded quaternionic operators.

\section{Preliminaries}
In this section, we recall some preliminaries   regarding   quaternionic $S$-functional calculus and slice regular functions (see \cite{Colombo-Sabadini-Struppa-11,Gentili-Stoppato-Struppa-13,Colombo-Gantner-19,Colombo-Gantner-Kimsey-18}).

Let $\mathbb H$ be the non-commutative (but associative) real algebra of quaternions with standard basis $\{1,\,i,\,j, \,k\}$ obeying the multiplication relations
$$i^2=j^2=k^2=ijk=-1.$$
 Every quaternion $q=x_0+x_1i+x_2j+x_3k $ $(x_0,x_1,x_2,x_3\in \mathbb{R})$ is spit  into the \textit{real} part ${\rm{Re}}\, (q)=x_0$ and the \textit{imaginary} part ${\rm{Im}}\, (q)=x_1i+x_2j+x_3k$. The \textit{conjugate} of $q\in \mathbb H$ is defined by $\bar{q}={\rm{Re}}\, (q)-{\rm{Im}}\, (q)$ and its \textit{modulus} is defined by $|q|=\sqrt{q\overline{q}}$.

The unit $2$-sphere of purely imaginary quaternions is given by
$$\mathbb S=\big\{q \in \mathbb H \ |\ q^2 =-1\big\}.$$
Denote by $\mathbb C_I$ the complex plane $\mathbb R \oplus I\mathbb R $, isomorphic to $ \mathbb C =\mathbb C_i$.

Let $X$ be a two-sided vector space over $\mathbb{H}$. A map $T:X \rightarrow X$ is said to be a right $\mathbb{H}$-linear operator
if
$$T (uq+ v) = (Tu)q +Tv,  \quad u, v\in X,  q\in \mathbb{H}.$$
The set of right $\mathbb{H}$-linear operators on $X$ is a two-sided vector space over $\mathbb{H}$  with respect to
the operations
$$(sT) (v) := s (Tv),  \quad (Ts) (v) :=  T (sv),  \quad  v\in X,  s\in \mathbb{H}.$$

 Unless specifically stated, we assume that from now until the end of the paper:

\textit{Given a two-sided quaternionic  Banach space $X$ over $\mathbb{H}$,   let $\mathcal{B}(X)$ be  the algebra of bounded right  $\mathbb{H}$-linear operators on $X$. The unit of $\mathcal{B}(X)$ is the identity operator  denoted by $\mathcal{I}_{X}$ or simply by $\mathcal{I}$ if no confusion may arise.}

\begin{definition}Let  $T\in \mathcal{B}(X)$. For $s\in \mathbb{H}$, set
$$Q_{s}(T):= T^{2}-2{\rm{Re}}\,(s)T+|s|^{2}\mathcal{I}.$$
  The  $S$-resolvent set of $T$ is defined as
$$ \rho_{S}(T)=\{s\in \mathbb{H}: Q_{s}(T) \ {\mbox {is   invertible in}}\ \mathcal{B}(X)\}.$$
and the  $S$-spectrum of $T$ is defined as
$$\sigma_{S}(T)= \mathbb{H}\setminus  \rho_{S}(T).$$
For $s\in  \rho_{S}(T)$, $Q_{s}(T)^{-1}\in \mathcal{B}(X)$  is called the pseudo-resolvent operator of $T$ at $s$.
\end{definition}

 \begin{definition}
 Let  $T\in \mathcal{B}(X)$. For $s\in \rho_{S}(T)$,  the left and right $S$-resolvent operator are defined as, respectively,
 $$S_{L}^{-1}(s,T)=Q_{s}(T)^{-1}(\overline{s}\mathcal{I}-T),$$
 and $$S_{R}^{-1}(s,T)=(\overline{s}\mathcal{I}-T)Q_{s}(T)^{-1}.$$
 \end{definition}

In a sense, the $S$-resolvent operator can be viewed as a Taylor series; see e.g.  \cite[Theorem 4.7.4]{Colombo-Sabadini-Struppa-11}.
\begin{theorem}\label{Spectrum-Taylor-theorem}
Let  $T\in \mathcal{B}(X)$ and $s\in \mathbb{H}$ with $\|T\|<|s|$. Then $\sum_{n=0}^{\infty}  T^n s^{-n-1}$  converge in the operator norm and
 \begin{equation}\label{Spectrum-Taylor}
 S_{L}^{-1}(s,T) =\sum_{n=0}^{\infty}  T^n s^{-n-1}. \end{equation}
\end{theorem}

Recall the $S$-spectral radius  for quaternionic linear  operators.
 \begin{definition}Let  $T\in \mathcal{B}(X)$.  The $S$-spectral radius of $T$ is defined as
 $$r_{S}(T)=\sup\{|s|: s\in \sigma_{S}(T)\}.$$
\end{definition}

 In terms of the $S$-spectral radius,   the Gelfand formula could be generalized into the setting of quaternionic linear  operators as in the classical case. See \cite[Theorem 4.12.6]{Colombo-Sabadini-Struppa-11}.
\begin{theorem}[$S$-spectral radius theorem]\label{spectral-radius}
Let $X$ be  a two-sided quaternionic Banach space, $T\in \mathcal{B}(X)$ and $r_{S}(T)$ be its $S$-spectral radius. Then
$$r_{S}(T)=\lim_{n\rightarrow\infty} \|T^{n}\|^{1/n}.$$
\end{theorem}

 \begin{lemma}\label{operator-formula}
 Let  $T\in \mathcal{B}(X)$ and $\Omega$ be a bounded slice Cauchy domain with $\sigma_{S} (T)\subset \Omega$. For every $I\in \mathbb{S}$, we have
$$T^{n}=\frac{1}{2\pi}\int_{\partial \Omega_{I} } S^{-1}_{L}(s,T)ds_{I}s^{n}, $$
and
$$T^{n}=\frac{1}{2\pi}\int_{\partial \Omega_{I} }s^{n}ds_{I} S^{-1}_{R}(s,T), $$
where $ds_{I}:=-Ids$ and $\Omega_{I}:=\Omega \cap \mathbb{C}_{I}$.
\end{lemma}
Here $\Omega$ is a slice Cauchy domain means that, for every $I\in \mathbb{S}$, the boundary $\partial \Omega_{I}$ of $\Omega_{I}$ is a finite union  of nonintersecting piecewise continuously differentiable Jordan curves in $\mathbb{C}_{I}$.

Now we    recall the definition of slice regularity.
\begin{definition} \label{de: regular}
Let $\Omega$ be a domain in $\mathbb H$. A function $f :\Omega \rightarrow \mathbb H$ is called left \emph{slice} \emph{regular} if, for all $ I \in \mathbb S$, its restriction $f_I$ to $\Omega_{I}$ is \emph{holomorphic}, i.e., it has continuous partial derivatives and satisfies
$$\bar{\partial}_I f(x+yI):=\frac{1}{2}\left(\frac{\partial}{\partial x}+I\frac{\partial}{\partial y}\right)f_I (x+yI)=0,\quad x+yI\in \Omega_I.$$
Similarly, the definition of right slice regular functions can be given.

Denote by $\mathcal{SR}_{L}(\Omega)$ and $\mathcal{SR}_{R}(\Omega)$ the set of left and right slice regular functions in $\Omega$, respectively.
 \end{definition}

  Based on Lemma \ref{operator-formula}  and  Definition \ref{de: regular},   quaternionic  $S$-functional calculus can be  formulated as follows.
  \begin{definition}
  Let  $T\in \mathcal{B}(X)$ and $\Omega$ be a bounded slice Cauchy domain with $\sigma_{S} (T)\subset \Omega$. Define quaternionic  $S$-functional calculus as
$$f(T) =\frac{1}{2\pi}\int_{\partial \Omega_{I} } S^{-1}_{L}(s,T)ds_{I}f(s), \quad f\in\mathcal{SR}_{L}(\Omega),$$
and
$$f(T)=\frac{1}{2\pi}\int_{\partial \Omega_{I} }s^{n}ds_{I} S^{-1}_{R}(s,T), \quad f\in\mathcal{SR}_{R}(\Omega). $$
\end{definition}

Two specific domains are introduced in the theory of slice regular functions.
\begin{definition} \label{de: domain}
Let $\Omega$ be a domain in $\mathbb H $.

1. $\Omega$ is called a \textit{slice domain}  if it intersects the real axis and if  $\Omega_I$  is a domain in $ \mathbb C_I $ for all $I \in \mathbb S$.

2. $\Omega$ is called an \textit{axially symmetric domain} if    $x + y\mathbb S\subseteq \Omega $ for all possible $x + yI \in \Omega$ with  $x,y \in \mathbb R $ and $I\in \mathbb S$.
\end{definition}

Let us recall more properties of $S$-spectrum for  quaternionic  linear operators, which shall be used in the sequel.
The first is the structure and compactness of  $S$-spectrum; see \cite[Theorems 4.8.6 and 4.8.11]{Colombo-Sabadini-Struppa-11}.

\begin{theorem}\label{Compactness}
Let   $T\in \mathcal{B}(X)$. Then the  $S$-spectrum of $\sigma_{S}(T)$ is  a nonempty, compact, and  axially symmetric set.
\end{theorem}

 Given $T\in \mathcal{B}(X)$, denote by $\mathcal{N}_{\sigma_{S}(T)}$  the set of all  functions $f\in \mathcal{SR}(\Omega)$ with $\sigma_{S} (T)\subset \Omega$ and preserving all slices, i.e. $f(\Omega_I)\subset \mathbb C_I$. With this notation, the $S$-spectral mapping theorem can be stated as follows \cite[Theorem 4.13.2]{Colombo-Sabadini-Struppa-11}.
\begin{theorem}\label{spectral-mapping}
Let  $T\in \mathcal{B}(X)$ and $f\in \mathcal{N}_{\sigma_{S}(T)}$. Then
$$ \sigma_{S}(f(T))= f(\sigma_{S}(T))=\{f(s):s\in \sigma_{S}(T) \}.$$
\end{theorem}

\section{Power-bounded quaternionic operators}
In this section, we generalize the concept of power-boundedness into the quaternionic operators and  study their properties.
  \begin{definition}Let $X$ be a two-sided vector space over $\mathbb{H}$ and  $T\in \mathcal{B}(X)$. The quaternionic operator  $T$ is called  power-bounded if
$$p(T):=\sup_{n\in \mathbb{N}}\|T^{n}\|<\infty.$$
\end{definition}

\begin{remark}\label{series-spetrum}
Let $T\in \mathcal{B}(X)$. From the Cauchy-Hadamard formula, the radius of convergence of the power series $\sum_{n=0}^{\infty} x^{n} \|T^n \|$ is  $1/(\limsup_{n\rightarrow\infty} \|T^{n}\|^{1/n})$. Combining this with  Theorem \ref{spectral-radius}, we infer that the series  $\sum_{n=0}^{\infty}  \|T^n \||s|^{-n-1}$ converges for $|s|>r_{S}(T)$ and then  so does $\sum_{n=0}^{\infty}  T^n s^{-n-1}$ in  (\ref{Spectrum-Taylor}). Hence,
$$S_{L}^{-1}(s,T) =\sum_{n=0}^{\infty}  T^n s^{-n-1}, \quad  |s|>r_{S}(T).$$
Furthermore, by Theorem \ref{spectral-radius}, the power-bounded quaternionic  operator $T$ has the property $$r_{S}(T)\leq1.$$
\end{remark}

In particular, (\ref{Spectrum-Taylor})  remains true for  power-bounded operator $T$ and $s\in \mathbb{H}$ with $|s|>1$ and then
\begin{eqnarray}\label{Kreiss-condition}
 \|S_{L}^{-1}(s,T)\|\leq\sum_{n=0}^{\infty}  \|T\|^n |s|^{-n-1}\leq p(T) \sum_{n=0}^{\infty}   |s|^{-n-1}= \frac{p(T)}{|s|-1},\quad |s|>1.
 \end{eqnarray}

As in the classical case, we can define the so-called  \textit{Kreiss condition} in the quaternionic setting.
\begin{definition}[Kreiss  condition]\label{Kreiss}
Let   $T\in \mathcal{B}(X)$. The quaternionic operator $T$  satisfies  the  Kreiss condition if
there exists a constant $C>0$ such that
 $$ \|S_{L}^{-1}(s,T)\|\leq\frac{C}{ |s|-1}, \quad |s|>1.$$  \end{definition}
Obviously, from (\ref{Kreiss-condition}) and Definition \ref{Kreiss}, all power-bounded quaternionic operators satisfy necessarily  the Kreiss condition.
Conversely,  the complex linear operator  $T$  is power-bounded when $T$ is defined on  a complex finite-dimensional Banach space   and satisfies the Kreiss condition (cf. \cite{LeVeque-Trefethen,Morton}), meanwhile this is not true for finite-dimensional sapces. Unfortunately, it is not clear  for the present authors in the quaternionic case.

We shall use the spherical Yosida approximation to give a equivalent characterization of  power-bounded quaternionic operators.
  \begin{definition}
Let $T\in \mathcal{B}(X)$. The  left and right spherical Yosida approximation of $T$ are defined as, respectively,
$$ \mathcal{Y}_{L}(s, T)=S_{L}^{-1}(s,T)s^{2}-s\mathcal{I}, \quad s\in\rho_{S}(T),$$
and
$$ \mathcal{Y}_{R}(s, T)=s^{2}S_{R}^{-1}(s,T)-s\mathcal{I}, \quad s\in\rho_{S}(T).$$
  \end{definition}
  Note that the spherical Yosida approximation of certain right linear closed operators appeared in \cite[Definition 4.5]{Colombo-Sabadini-11}.

  For $T\in \mathcal{B}(X)$  and   $s\in \mathbb{H}$ with $\|T\|<|s|$,  it follows from Theorem \ref{Spectrum-Taylor-theorem} that
 $$\| \mathcal{Y}_{L}(s, T)-T\|=\|\sum_{n=1}^{\infty}  T^{n+1} s^{-n}\|\leq \sum_{n=1}^{\infty} \| T\|^{n+1}|s|^{-n}=\frac{\| T\|^{2}}{|s|-\| T\|}\rightarrow 0, \ as \ |s|\rightarrow \infty.$$

The spherical Yosida approximation can be rewritten in the following form.
\begin{lemma}\label{Yosida-approximation-form}
For $T\in \mathcal{B}(X)$, it holds that
$$ \mathcal{Y}_{L}(s, T)=TS_{L}^{-1}(s,T)s,\quad s\in\rho_{S}(T).$$
\end{lemma}
\begin{proof}
Note that the commutative relation  $Q_{s}(T)T=Q_{s}(T)T$ implies that $Q_{s}(T)^{-1}T=Q_{s}(T)^{-1}T$ for $s\in\rho_{S}(T)$. Hence,
\begin{eqnarray*}
   \mathcal{Y}_{L}(s, T)
&=&  Q_{s}(T)^{-1}((\overline{s}\mathcal{I}-T)s-Q_{s}(T))s
\\
&=& Q_{s}(T)^{-1}( T\overline{s} -T^{2})s
\\
&=& TQ_{s}(T)^{-1}( \overline{s}\mathcal{I} -T)s
\\
&=& TS_{L}^{-1}(s,T)s,
\end{eqnarray*}
 as desired.
\end{proof}

Based on the observation in Lemma \ref{Yosida-approximation-form}, we define,  for $n\in \mathbb{N}$,
$$ \mathcal{Y}_{L}^{n}(s, T)=T^{n}S_{L}^{-n}(s,T)s^{n},\quad s\in\rho_{S}(T),$$
where
$$S_{L}^{-n}(s,T)=Q_{s}(T)^{-n}\sum_{m=0}^{n} \begin{pmatrix}n \\ m\end{pmatrix}(-T)^{m} \overline{s}^{n-m}.$$
Similarly, we can define,  for $n\in \mathbb{N}$,
$$ \mathcal{Y}_{R}^{n}(s, T)=s^{n}S_{R}^{-n}(s,T)T^{n},\quad s\in\rho_{S}(T),$$
where
$$S_{R}^{-n}(s,T)=\sum_{m=0}^{n} \begin{pmatrix}n \\ m\end{pmatrix}\overline{s}^{n-m}(-T)^{m} Q_{s}(T)^{-n}.$$

\begin{lemma}\label{spectrum-n}
For $T\in \mathcal{B}(X)$ and $n\in \mathbb{N}$, it holds that
$$S_{L}^{-n}(s,T)=\sum_{m=0}^{\infty} \begin{pmatrix}m+n-1 \\ n-1\end{pmatrix} T^{m}s^{-(m+n)}, \quad  |s|>r_{S}(T).$$
\end{lemma}
\begin{proof}
The desired formula  is equivalent to proving the following identity, for all  $n\in \mathbb{N}$,
 \begin{eqnarray}\label{formula-n}
 Q_{s}(T)^{n}\sum_{m=0}^{\infty} \begin{pmatrix}m+n-1 \\ n-1\end{pmatrix} T^{m}s^{-(m+n)}= \sum_{m=0}^{n} \begin{pmatrix}n \\ m\end{pmatrix}(-T)^{m} \overline{s}^{n-m}, \quad |s|>r_{S}(T). \end{eqnarray}
We shall show its validity by mathematical induction.  For $n=1$, (\ref{formula-n}) reduces into
$$Q_{s}(T)\sum_{m=0}^{\infty}  T^{m}s^{-(m+1)}= \sum_{m=0}^{1} \begin{pmatrix}1 \\ m\end{pmatrix}(-T)^{m} \overline{s}^{1-m} =\overline{s}\mathcal{I}-T, \quad |s|>r_{S}(T),$$
which has been   observed exactly by Remark \ref{series-spetrum}. More details can be founded in \cite[Theorem 4.7.4]{Colombo-Sabadini-Struppa-11}.
Now assume that the formula in  (\ref{formula-n}) holds for some positive integer $n$. Let us prove the same formula
  holds for $n+1$.

From the  combinatorial identity
$$   \sum_{l=0}^{m} \begin{pmatrix}m+n-1  \\ n-1\end{pmatrix} =\begin{pmatrix}m+n  \\ n\end{pmatrix}, \quad n\in \mathbb{N},$$
we get, for $s\in \mathbb{H}$ with $|s|>r_{S}(T)$,
\begin{eqnarray*}
  & &  \sum_{m=0}^{\infty} \begin{pmatrix}m+n  \\ n\end{pmatrix} T^{m}s^{-(m+n+1)}
     \\
&=&  \sum_{m=0}^{\infty} \sum_{l=0}^{m} \begin{pmatrix}m+n-1  \\ n-1\end{pmatrix} T^{m}s^{-(m+n+1)}
\\
&=& \sum_{m=0}^{\infty} \sum_{l=0}^{\infty} \begin{pmatrix}m+n-1  \\ n-1\end{pmatrix} T^{m+l}s^{-(m+n+l+1)}
\\
&=&  \sum_{m=0}^{\infty}\begin{pmatrix}m+n-1  \\ n-1\end{pmatrix}T^{m} \Big( \sum_{l=0}^{\infty}  T^{l}s^{-(l+1)} \Big) s^{-(m+n)}.
\end{eqnarray*}
Hence, under  the result of (\ref{formula-n}) for $n=1$ and $n$ and keeping in mind of the commutative relationships   $Q_{s}(T)T=Q_{s}(T)T$ and $s\overline{s}=\overline{s}s$, we obtain
\begin{eqnarray*}
& &    Q_{s}(T)^{n+1}\sum_{m=0}^{\infty} \begin{pmatrix}m+n \\ n\end{pmatrix} T^{m}s^{-(m+n+1)}
\\
&=& Q_{s}(T)^{n+1}  \sum_{m=0}^{\infty} \begin{pmatrix}m+n-1 \\ n-1\end{pmatrix} T^{m}  S_{L}^{-1}(s,T) s^{-(m+n)}
\\
&=& Q_{s}(T)^{n}\sum_{m=0}^{\infty} \begin{pmatrix}m+n-1 \\ n-1\end{pmatrix} T^{m} (\overline{s}\mathcal{I}-T)s^{-(m+n)}
\\
&=&  Q_{s}(T)^{n}\sum_{m=0}^{\infty} \begin{pmatrix}m+n-1 \\ n-1\end{pmatrix} T^{m}s^{-(m+n)}  \overline{s}
-TQ_{s}(T)^{n}\sum_{m=0}^{\infty} \begin{pmatrix}m+n-1 \\ n-1\end{pmatrix} T^{m}s^{-(m+n)}
\\
&=&\sum_{m=0}^{n}  \begin{pmatrix}n \\ m\end{pmatrix} (-T)^{m} \overline{s}^{n+1-m}
+\sum_{m=0}^{n}\begin{pmatrix}n \\ m\end{pmatrix}(-T)^{m+1} \overline{s}^{n-m}
\\
&=&\overline{s}^{n+1}+\sum_{m=1}^{n}  \begin{pmatrix}n \\ m\end{pmatrix} (-T)^{m} \overline{s}^{n+1-m}
+\sum_{m=1}^{n}\begin{pmatrix}n \\ m-1\end{pmatrix}(-T)^{m} \overline{s}^{n+1-m}+(-T)^{n+1}
\\
&=& \overline{s}^{n+1} +\sum_{m=1}^{n} \Big(\begin{pmatrix}n \\ m\end{pmatrix}+\begin{pmatrix}n \\ m-1\end{pmatrix}\Big)(-T)^{m} \overline{s}^{n+1-m}+(-T)^{n+1}
\\
&=& \overline{s}^{n+1}+\sum_{m=1}^{n} \begin{pmatrix}n+1 \\ m\end{pmatrix}(-T)^{m} \overline{s}^{n+1-m}+(-T)^{n+1}
\\
&=&   \sum_{m=0}^{n+1} \begin{pmatrix}n+1 \\ m\end{pmatrix}(-T)^{m} \overline{s}^{n+1-m},
\end{eqnarray*}
which completes the proof.
\end{proof}

Now  we prsent  a discrete analog of Theorem \ref{Hille-Yosida-Phillips} for power-bounded quaternionic operators. Its classical complex result is due to   Nevanlinna \cite[Theorem 2.7.1]{Nevanlinna}, which    is essentially same with the observation in \cite{Gibson}.

 \begin{theorem}\label{Nevanlinna-H}
Let $T\in \mathcal{B}(X)$.  The following three statements are equivalent:

(i) $T$ is a power-bounded  quaternionic operator;

(ii)there exits some constant $C>0$ such that
\begin{eqnarray}\label{discrete-condition}
\|\mathcal{Y}_{L}^{n}(s, T)\|\leq \frac{C}{(1-\frac{1}{|s|})^{n}}, \quad |s|>1, n\in \mathbb{N}; \end{eqnarray}

(iii)there exits some constant $C>0$ such that
$$\|\mathcal{Y}_{R}^{n}(s, T)\|\leq \frac{C}{(1-\frac{1}{|s|})^{n}}, \quad |s|>1, n\in \mathbb{N}.$$
\end{theorem}
\begin{proof}
We only prove the equivalence of (i) and (ii).

 (i) $\Rightarrow$ (ii)
From Lemma \ref{spectrum-n}, we have
\begin{eqnarray}\label{Yosida-approximation}
 \mathcal{Y}_{L}^{n}(s, T)=T^{n}\sum_{m=0}^{\infty} \begin{pmatrix}m+n-1 \\ n-1\end{pmatrix} T^{m}s^{-m}, \quad |s|>r_{S}(T).\end{eqnarray}
 Hence, for $s\in\mathbb{H}$ with $|s|>1$ $(\geq r_{S}(T)$ by Remark \ref{series-spetrum}$)$, we get
 \begin{eqnarray*}
   \|\mathcal{Y}_{L}^{n}(s, T)\|
&\leq &  \sum_{m=0}^{\infty} \begin{pmatrix}m+n-1 \\ n-1\end{pmatrix}\|T\|^{m+n}|s|^{-m}
\\
& \leq & p(T) \sum_{m=0}^{\infty} \begin{pmatrix}m+n-1 \\ n-1\end{pmatrix} |s|^{-m}
\\
&=& \frac{p(T)}{(1-|s|^{-1})^{n}}.
\end{eqnarray*}
 (ii) $\Rightarrow$ (i)
Under the condition of (\ref{discrete-condition}), we obtain
$$\limsup_{|s|\rightarrow \infty}\|\mathcal{Y}_{L}^{n}(s, T)\|\leq C.$$
 For $|s|>r_{S}(T)$, it follows from (\ref{Yosida-approximation}) that
  \begin{eqnarray*}
  \|\mathcal{Y}_{L}^{n}(s, T)-T^{n}\|
&\leq &  \sum_{m=1}^{\infty} \begin{pmatrix}m+n-1 \\ n-1\end{pmatrix} \|T \|^{m+n}|s|^{-m}
\\
&=& \|T \|^{n} (\frac{1}{(1-\|T\||s|^{-1})^{n}}-1)\rightarrow 0, \ as \ |s|\rightarrow \infty.
\end{eqnarray*}
 Hence, combining those two results above, we get $\|T ^{n}\|\leq C$, as desired.

 \end{proof}


Enclosing this section, we give a sufficient condition of the power-boundness for quaternionic operators.
 \begin{theorem}\label{sufficient-condition}
Let $\mathbb{B}$ denote  the open unit ball $\{q\in \mathbb{H}: |q|<1\}$ and $T\in \mathcal{B}(X)$. If $\Gamma_{S} (T) :=\sigma_{S} (T)\cap \partial \mathbb{B}\subseteq \{1\}$ and there exist    constants $C>0$ and $0<\alpha\leq 1$ such that
  \begin{eqnarray}\label{Ritt-condition-2}\|S_{L}^{-2}(s,T)\|\leq\frac{C}{|s-1|^{1+\alpha}}, \quad |s|>1,  \end{eqnarray}
 then $T$ is  power-bounded.
 \end{theorem}
\begin{proof}Under the condition in  (\ref{Ritt-condition-2}), we have  $\sigma_{S} (T)\subseteq  \overline{\mathbb{B}}$.
Hence,   by Lemma \ref{operator-formula},
 $$T^{n}=\frac{1}{2\pi}\int_{\partial  U_{I} } S^{-1}_{L}(s,T)ds_{I}s^{n},  \quad I\in \mathbb{S}, $$
 where $U=r\mathbb{B}$ with $r>1$.

Through a partial integration, we get
$$T^{n}=\frac{1}{2\pi(n+1)}\int_{\partial  U_{I}  } S^{-2}_{L}(s,T)ds_{I}s^{n+1}. $$
Hence, it follows  from the condition  (\ref{Ritt-condition-2})  that, for $0<\alpha\leq 1$,
\begin{eqnarray}\label{Forelli-Rudin-0}
\|T^{n}\|\leq\frac{C}{2\pi(n+1)}\int_{\partial  U_{I} } \frac{|s|^{n+1}}{|s-1|^{1+\alpha}} |ds_{I}|
=\frac{Cr^{n+1-\alpha}}{2\pi(n+1)}\int_{-\pi}^{\pi}   \frac{ d\theta}{|1-r^{-1}e^{I \theta}|^{1+\alpha}}. \end{eqnarray}
The Forelli-Rudin  estimate (see e.g. \cite[Theorem 1.12]{Zhu}) gives that, for some constant $C_{0}>0$,
$$\int_{-\pi}^{\pi}   \frac{d\theta }{|1-r^{-1}e^{I \theta}|^{1+\alpha}}\leq \frac{C_{0}}{(1-r^{-2})^{\alpha}}, \quad as \ r\rightarrow 1^{+}.  $$
 This  along with (\ref{Forelli-Rudin-0}) produces that, for some constant $C>0$ and all $n\in \mathbb{N}$,
$$\|T^{n}\|\leq\frac{Cr^{n+1+\alpha}}{(n+1)(r^{2}-1)^{\alpha}}, \quad as \ r\rightarrow 1^{+}.$$
Taking $r=\sqrt{1+\frac{1}{n^{\frac{1}{\alpha}}}}$, we arrive at
 $$ \|T^{n}\|\leq  \frac{ Cx_{n}}{ 1+\frac{1}{n} }\leq Cx_{n},\ as \ n\rightarrow \infty, $$
 where $ x_{n}$ is bounded above for $0<\alpha\leq 1$ and given by
  $$\log x_{n}=\frac{1}{2}(n+1+\alpha)\log (1+\frac{1}{n^{\frac{1}{\alpha}}})\sim \frac{1}{2n^{\frac{1}{\alpha}-1}},   \ as \ n\rightarrow \infty,$$
which completes the proof.
\end{proof}

\section{Katznelson-Tzafriri theorem}

\subsection{The complex case}
 In this subsection,  we recall primarily the classical Katznelson-Tzafriri theorem and some related results.

Let $\mathbb{D}$ be the unit open disk of complex plane $\mathbb{C}$ and $\mathbb{T}$ be its boundary. Denote by $W^{+}(\mathbb{D})$ the Banach algebra of all holomorphic functions of the form $f(z)=\sum_{n=0}^{\infty}  a_{n}z^n$ $(z\in  \mathbb{D})$ with $\| f \|_{W}=\sum_{n=0}^{\infty}  |a_{n}|<+\infty$.  The boundary functions of $W^{+}(\mathbb{D})$ form a closed subalgebra of the Wiener algebra $W(\mathbb{T})$  of all functions  $g(z)=\sum_{n=-\infty}^{\infty}  a_{n}z^n$ on $\mathbb{T}$ with $\| g \|_{W}=\sum_{n=-\infty}^{\infty}  |a_{n}|<+\infty$.

Let  $T$ be a $\mathbb{C}$-linear operator  on a complex Banach space. As usual, denote by $\sigma(T)$ the   spectrum of $T$.  In  1986, Katznelson and    Tzafriri proved a  celebrated  result \cite[Theorem 5]{Katznelson-Tzafriri} for  power-bounded  operators.
\begin{theorem}\label{Katznelson-Tzafriri-theorem}
Let $T$ be a power-bounded $\mathbb{C}$-linear operator  on a complex Banach space and $f\in W^{+}(\mathbb{D})$. If $f$   is of spectral synthesis in $ W(\mathbb{T})$ with respect to peripheral spectrum $\Gamma (T):=\sigma(T)\cap \mathbb{T}$, then
$$\lim_{n\rightarrow\infty} \|T^{n}f(T)\|=0.$$
\end{theorem}
The assumption of spectral synthesis in Theorem \ref{Katznelson-Tzafriri-theorem} means that there is a functions sequence $\{g_{n}\}_{n\in \mathbb{N}} \subseteq W(\mathbb{T})$
 such that each $g_{n}$   vanishes on a neighbourhood of $\Gamma (T)$ in $\mathbb{T}$ and $\lim_{n\rightarrow\infty} \|g_{n}-f \|_{W}=0.$

The most popular version of Katznelson-Tzafriri theorem  is known as  the special case of  $f(z)=z-1$; see \cite[Theorem 1]{Katznelson-Tzafriri}.
 \begin{theorem}\label{Katznelson-Tzafriri-popular}
Let $T$ be a power-bounded $\mathbb{C}$-linear operator  on a complex Banach space. Then
$$\lim_{n\rightarrow\infty} \|T^{n}-T^{n+1} \|=0 \Leftrightarrow \Gamma (T) \subseteq \{1\}.$$
\end{theorem}

Regarding  proofs of Theorem \ref{Katznelson-Tzafriri-popular}, there  exit historically several other methods   (cf. \cite{Allan,Allan-Ransford,Vu}).  In fact,  a partial result  of Theorem \ref{Katznelson-Tzafriri-popular} for general Banach algebras was  obtained earlier in 1983 by Esterle  \cite[Theorem 9.1]{Esterle} saying that  if $a$ is an element of norm $1$ in a unital Banach algebra with its   spectrum   is $\{1\}$, then $\lim_{n\rightarrow\infty} \|a^{n}-a^{n+1} \|=0$.

  The origin of  Theorem \ref{Katznelson-Tzafriri-popular}  traces probably back to the so-called \textit{zero-two law: for any positive contraction    $T$    on a $L_{1}$-space, either $\|T^{n}-T^{n+1}\|=2$ for all $n\in \mathbb{N}$ or  $\lim_{n\rightarrow\infty} \|T^{n}-T^{n+1}\|=0$.} See for example \cite{Ornstein-Sucheston}.  The Katznelson-Tzafriri theorem has been obtained   substantial subsequent interest,  especially, which  turns into  one of  cornerstones in the asymptotic theory of operator semigroups.  For the history of Katznelson-Tzafriri theorem and related results,   see surveys e.g. \cite{Batty,Batty-Seifert} and \cite{Chill-Tomilov}.

\subsection{The quaternionic case}
In this subsection, we   shall establish a quaternionic analogue of the Katznelson-Tzafriri theorem (Theorem \ref{Katznelson-Tzafriri-popular}) by  the $S$-spectrum.

\begin{theorem}\label{Katznelson-Tzafriri-quaternion}
Let $T$ be a  power-bounded quaternionic operator    on a  two-sided quaternionic Banach space $X$. Then
$$\lim_{n\rightarrow\infty} \|T^{n}-T^{n+1} \|=0 \Leftrightarrow   \Gamma_{S} (T) \subseteq \{1\},$$
where $\Gamma_{S} (T)$ is defined as in Theorem \ref{sufficient-condition}.
\end{theorem}

As a direct  consequence, Theorem \ref{Katznelson-Tzafriri-quaternion}  gives the quaternionic version of  the classical result of Gelfand.
See  \cite{Gelfand} or \cite{Allan-Ransford,Zemanek} for the classical case.

\begin{corollary}
Let $T$ be a right $\mathbb{H}$-linear operator  on a  two-sided Banach space over $\mathbb{H}$. If  $T$ is doubly  power-bounded, i.e. $\sup_{n\in \mathbb{Z}}\|T^{n}\|<\infty$, and  $ \sigma_{S} (T) = \{1\}$, then $T=\mathcal{I}$.
\end{corollary}

\begin{corollary}[Gelfand]
Let $T$ be a $\mathbb{C}$-linear  operator  on a complex Banach space. If  $T$ is doubly  power-bounded, i.e. $\sup_{n\in \mathbb{Z}}\|T^{n}\|<\infty$, and  $ \sigma (T) = \{1\}$, then $T=\mathcal{I}$.
\end{corollary}

Following the idea of  Katznelson and Tzafriri in \cite{Katznelson-Tzafriri}, we   use  arguments of harmonic analysis to prove Theorem \ref{Katznelson-Tzafriri-quaternion}. Before do this, we need to fix some notation which shall be used in the proof.

Given two sequences $\{a_{n}\}, \{b_{n}\}\subset \mathbb{H}$, the convolution sequence of $\{a_{n}\} $ and $ \{b_{n}\}$ is  defined by
$\{\sum_{m=-\infty}^{+\infty} a_{n-m}b_{m}\}$, denoted  by $\{a_{n}\}\star\{ b_{n}\}$.
Being different from the classical  complex case, the convolution  of two quaternionic sequences defined above is noncommutative.

For $\{a_{n}\} \subset \mathbb{H}$, define
\begin{eqnarray*}
\|\{a_{n}\}\|_{p}=
\left\{
\begin{array}{lll}
   \sum_{n\in \mathbb{Z}} |a_{n}|^{p})^{\frac{1}{p}},     & 0<p<+\infty,
\\
\sup_{n\in \mathbb{Z}} |a_{n}|,   &p=+\infty.
\end{array}
\right.
\end{eqnarray*}
 As in the classical  case, it holds that
\begin{eqnarray}\label{technical-convolution}
\|\{a_{n}\}\star\{ b_{n}\}\|_{\infty} \leq \|\{a_{n}\} \|_{\infty} \|\{b_{n}\} \|_{1}.\end{eqnarray}

Now we recall a technical  lemma \cite[Lemma 4.2.3]{Nevanlinna}.
 \begin{lemma}\label{technical-lemma}
For any $\epsilon>0$, there exists a nonnegative $\chi_{\epsilon} \in C^{2}([-\pi,\pi]) $  such that
\begin{eqnarray*}
\chi_{\epsilon}(\theta)=
\left\{
\begin{array}{lll}
   1,     &\mathrm {if} \  |\theta|\leq  \epsilon/2,
\\
0,   &\mathrm {if} \ |\theta| \geq  \epsilon,
\end{array}
\right.
\end{eqnarray*}
and
$$\sum_{n=-\infty}^{+\infty} |\widehat{\chi_{\epsilon}}(n+1)- \widehat{\chi_{\epsilon}}(n)| <\epsilon, $$
where $\widehat{\chi_{\epsilon}}(n)$ are  Fourier coefficients given by
$$\widehat{\chi_{\epsilon}}(n)=\frac{1}{2\pi}\int_{-\pi}^{\pi}\chi_{\epsilon}(\theta)e^{-in\theta} d\theta.$$
\end{lemma}

It is well-known that, for the $C^{1}$ function $f:[-\pi,\pi] \rightarrow \mathbb{C}$,
$$f(\theta)=\sum_{n=-\infty}^{+\infty}\widehat{f}(n) e^{in\theta}. $$
The quaternion valued  function  $f$ can be split into
$$f=F+jG, \quad F, G \in \mathbb{C}.$$
Hence,
$$f(\theta)=\sum_{n=-\infty}^{+\infty}(\widehat{F}(n) +j \widehat{G}(n) )e^{in\theta}. $$
In  view of this, for the $C^{1}$ function $f:[-\pi,\pi] \rightarrow \mathbb{H}$ and $I\in \mathbb{S}$,  its  left and right Fourier coefficients  are defined by, respectively,
$$\mathfrak{F}_{L,I}\{f\}(n)=\frac{1}{2\pi}\int_{-\pi}^{\pi}e^{-In\theta} f(\theta)d\theta, \quad n\in \mathbb{Z},$$
$$\mathfrak{F}_{R,I}\{f\}(n)=\frac{1}{2\pi}\int_{-\pi}^{\pi}f(\theta)e^{-In\theta} d\theta, \quad n\in \mathbb{Z}.$$
Consequently,
$$f(\theta)=\sum_{n=-\infty}^{+\infty}e^{In\theta}\mathfrak{F}_{L,I}\{f\}(n)=\sum_{n=-\infty}^{+\infty}\mathfrak{F}_{R,I}\{f\}(n) e^{In\theta}. $$

For  $n\in \mathbb{Z}$ and $C^{1}$ functions $f,g:[-\pi,\pi] \rightarrow \mathbb{H}$, we have
\begin{eqnarray*}
 \mathfrak{F}_{R,I}\{fg\}(n)
&= &  \frac{1}{2\pi}\int_{-\pi}^{\pi}f(\theta)g(\theta)e^{-In\theta} d\theta, \quad
\\
&=& \sum_{m=-\infty}^{+\infty}\mathfrak{F}_{R,I}\{f\}(m) \frac{1}{2\pi}\int_{-\pi}^{\pi} e^{Im\theta} g(\theta)e^{-In\theta} d\theta.
\end{eqnarray*}
If further  $g$ takes values in $\mathbb{C}_{I}$, then
$$\mathfrak{F}_{R,I}\{fg\}(n)=\sum_{m=-\infty}^{+\infty}\mathfrak{F}_{R,I}\{f\}(m) \mathfrak{F}_{R,I}\{f\}(n-m),$$
i.e.,
\begin{eqnarray}\label{technical-convolution-2}
\{\mathfrak{F}_{R,I}\{fg\}(n)\}= \{\mathfrak{F}_{R,I}\{f \}(n)\} \star \{\mathfrak{F}_{R,I}\{g\}(n)\}.\end{eqnarray}

Now we are in a position to prove the quaternionic version of  Katznelson-Tzafriri theorem.
\begin{proof}[Proof of Theorem \ref{Katznelson-Tzafriri-quaternion}]
$\Rightarrow$)
Let $s\in  \Gamma_{S} (T)$. Note  that the function $f(q)=q^{n}-q^{n+1}$ is a slice regular  on $\mathbb{H}$ and preserves each slice. By Theorem \ref{spectral-mapping}, it follows that $f(s)\in     \sigma_{S}(f(T))$. Hence, from  Theorem \ref{spectral-radius} we have
$$\|T^{n}-T^{n+1} \|=\|f(T) \|\geq r_{S}(f(T))\geq |f(s)|=|s-1|.$$
Consequently,  the condition $\lim_{n\rightarrow\infty} \|T^{n}-T^{n+1} \|=0$ forces that $s=1$.

$\Leftarrow$)
From Remark  \ref{series-spetrum}, it holds that  $r_{S}(T)\leq1$ for the power-bounded quaternionic  operator $T$. Hence,   by Lemma \ref{operator-formula}, we get, for any $I\in \mathbb{S}$,
 $$T^{n}=\frac{1}{2\pi}\int_{\partial  U_{I} } S^{-1}_{L}(s,T)ds_{I}s^{n},   $$
 where $U=r\mathbb{B}$ with $r>1$.

Then
$$ T^{n}=\frac{1}{2\pi}\int_{-\pi}^{\pi} S^{-1}_{L}(re^{I\theta},T) r^{n+1} e^{I(n+1)\theta}d\theta, $$
which implies
$$r^{-(n+1)}T^{n}(r^{-1}T-1)=\frac{1}{2\pi}\int_{-\pi}^{\pi} S^{-1}_{L}(re^{I\theta},T) (e^{I\theta}-1)e^{I(n+1)\theta}d\theta. $$
Denote
$$I_{1}=\frac{1}{2\pi}\int_{-\pi}^{\pi} \psi(\theta) \phi(\theta)e^{I(n+1)\theta}d\theta=\mathfrak{F}_{R,I}\{\psi  \phi\}(-n-1),$$
and
$$I_{2}=\frac{1}{2\pi}\int_{-\pi}^{\pi} \varphi(\theta)e^{I(n+1)\theta}d\theta
=\frac{1}{2\pi(n+1)}\int_{-\pi}^{\pi} \varphi'(\theta) Ie^{I(n+1)\theta}d\theta,$$
where
$$\psi(\theta)=S^{-1}_{L}(re^{I\theta},T),$$
$$\phi(\theta)=(e^{I\theta}-1)\chi_{\epsilon}(\theta), $$
$$\varphi(\theta)=\psi(\theta)(e^{I\theta}-1)(1-\chi_{\epsilon}(\theta)).$$
Here $\chi_{\epsilon}$ is given as in  Lemma \ref{technical-lemma}.

Note that
$$\mathfrak{F}_{R,I}\{\phi \}(n)=\mathfrak{F}_{R,I}\{\chi_{\epsilon} \}(n-1)-\mathfrak{F}_{R,I}\{\chi_{\epsilon} \}(n),$$
which implies  under the condition of Lemma \ref{technical-lemma} that
$$\|\{\mathfrak{F}_{R,I}\{\phi \}(n)\}\|_{1}<\epsilon.$$

Noticing also that $$\|\mathfrak{F}_{R,I}\{\psi\}(n)\| \leq\|T^{n-1}r^{-n}\|<p(T),$$
and $\phi(\theta)\in \mathbb{C}_{I}$, we have, by (\ref{technical-convolution}) and (\ref{technical-convolution-2}),
$$\|I_{1}\|\leq\|\{\mathfrak{F}_{R,I}\{\psi  \}(n) \star \mathfrak{F}_{R,I}\{\phi  \}(n)\}\|_{\infty}
\leq\|\{\mathfrak{F}_{R,I}\{\psi  \}(n)\} \|_{\infty}\|\{\mathfrak{F}_{R,I}\{\phi \}(n)\}\|_{1}<p(T)\epsilon.$$

Denote
$$K_{\epsilon}=\{re^{I\theta}: \frac{\epsilon}{2}\leq|\theta|\leq\pi, 1\leq r\leq2, I\in \mathbb{S}\}.$$
Now the operator-valued function $ S^{-1}_{L}(s,T)$ is right slice regular  in $K_{\epsilon}$, then $\varphi'(\theta)$ is  $C^{1}$ and hence   bounded in the compact set $K_{\epsilon}$.
Furthermore,   $\varphi(\theta)\equiv0$ for $|\theta|\leq  \epsilon/2$. Consequently, there exits a constant $C(\epsilon)$ such that, for all $r\in (1,2)$,
 $$\|\varphi'(\theta)\|\leq C(\epsilon),$$
which implies that
$$\|I_{2}\| \leq \frac{C(\epsilon)}{n+1}. $$
In summary, we get
$$\|r^{-(n+1)}T^{n}(r^{-1}T-1)\| \leq \|I_{1}\| +\|I_{2}\|<p(T)\epsilon+\frac{C(\epsilon)}{n+1}.$$
By letting $r\rightarrow1$ and $n\rightarrow \infty$, we obtain
$$\lim_{n\rightarrow\infty} \|T^{n}-T^{n+1} \|\leq p(T)\epsilon,$$
the desired result follows from the arbitrariness of $\epsilon$.\end{proof}
\bigskip
\textbf{Statements and Declarations}\\
\textbf{Author contributions}
All authors have contributed equally to all aspects of this manuscript and have reviewed its final draft.
\\
\textbf{Conflict of interest}
On behalf of all authors, the corresponding author states that there is no conflict of interest.
\\
\textbf{Data availability} No datasets are analysed or generated in this article.




\vskip 10mm
\end{document}